\documentclass[a4paper, 12pt]{article}
\usepackage{amsmath}
\usepackage{amssymb}
\usepackage{amsthm}

\newtheorem{theorem}{Theorem}[section] 
\newtheorem{lemma}[theorem]{Lemma}     
\newtheorem{corollary}[theorem]{Corollary}
\newtheorem{proposition}[theorem]{Proposition}

\providecommand{\ch}{\mathop{\rm ch}\nolimits}

\begin{document}

\title
{A formula on Stirling numbers of
the second kind and its application to the unstable $K$-theory 
of stunted complex projective spaces}

\author
{Osamu Nishimura}

\date{\empty}

\maketitle

\begin{abstract}
A formula on Stirling numbers of the second kind $S(n, k)$ is proved.
As a corollary, 
for odd $n$ and even $k$,
it is shown that $k!S(n, k)$ is a positive multiple
of the greatest common divisor of $j!S(n, j)$ for $k+1\leq j\leq n$.
Also, as an application to algebraic topology,
some isomorphisms of unstable $K^1$-groups
of stunted complex projective spaces are deduced.
\end{abstract}

\section{Introduction} 
\label{intro}
Let $S(n, k)$ be the Stirling number of the second kind,
which is the number of partitions of a set
consisting of $n$ elements
into $k$ pairwise disjoint non-empty subsets,
and let $\tilde{S}(n, k)=k!S(n, k)$.
Let $R$ be the function defined on non-negative integers by
\[
 R(m)=\frac{2(2^{2m+2}-1)B_{2m+2}}{m+1}
\]
where $B_s$ is the $s$th Bernoulli number.
The main purpose of this paper is to show the following theorem.
\begin{theorem}
\label{main}
For any odd integer $n\geq 3$ and any positive even integer $k$
such that $k<n$,
the equation
\[
\tilde{S}(n, k)
=\frac{1}{2}\sum_{h=0}^{(n-k-1)/2}R(h)
\binom{k+2h}{k-1}\tilde{S}(n, k+2h+1)
\]
holds.
\end{theorem}
For example, we have
\begin{align*}
\tilde{S}(7, 4)&=\frac{1}{2}\left(R(0)\binom{4}{3}\tilde{S}(7, 5)
+R(1)\binom{6}{3}\tilde{S}(7, 7)\right) \\
&=\frac{1}{2}\left(1\times 4\times 16800
-\frac{1}{2}\times 20\times 5040\right)=8400.
\end{align*}

For this purpose, 
first we investigate properties 
of the function $R$,
which we call 
{\em the recurrence weight (of Stirling numbers of the second kind)\/}
in this paper. 
It has representations
\begin{align*}
R(m)
&=\sum_{j=1}^m \left(-\frac{1}{2}\right)^j
\sum_{(I_1, \cdots, I_j)\in \tilde{E}_{m, j}}
\binom{2I_2}{2I_1}\binom{2I_3}{2I_2}\cdots\binom{2I_j}{2I_{j-1}} \\
&=\sum_{j=0}^m \left(-\frac{1}{2}\right)^j
\sum_{(I_0, I_1, \cdots, I_j)\in \tilde{O}_{m, j}}
\binom{2I_1-1}{2I_0-1}\binom{2I_2-1}{2I_1-1}\cdots\binom{2I_j-1}{2I_{j-1}-1}
\end{align*}
for any positive integer $m$ where
\begin{align*}
\tilde{E}_{m, j}&=\{(I_1, \cdots, I_j)\in \mathbb{Z}^j\ 
\vert\ 0<I_1<I_2<\cdots <I_j=m\}, \\
\tilde{O}_{m, j}&=\{(I_0, I_1, \cdots, I_j)\in \mathbb{Z}^{j+1}\ 
\vert\ 0<I_0<I_1<\cdots <I_j=m+1\}, 
\end{align*}
and
\[
\binom{2I_2}{2I_1}\binom{2I_3}{2I_2}\cdots\binom{2I_j}{2I_{j-1}},
\quad 
\binom{2I_1-1}{2I_0-1}\binom{2I_2-1}{2I_1-1}\cdots\binom{2I_j-1}{2I_{j-1}-1}
\]
are to be interpreted as $1$ for $j=1$ and for $j=0$ respectively.
Thus,
the recurrence weight $R$ also has combinatorial nature,
from which some simple recurrence formulae are derived.
We also describe recurrence formulae for 
Bernoulli numbers $B_{2m+2}$
which are derived from those for the recurrence weight $R$,
because they may be of independent interest.

Next, 
we consider the identity 
\[
 \tilde{S}(n, k)
=\sum_{j=k}^{n}(-1)^{n-j}\binom{j-1}{k-1}
\tilde{S}(n, j).
\]
(See Quaintance and Gould~\cite[(9.18)]{QG_2016}.)
In particular, 
if $n$ and $k$ are as in Theorem~{\rm\ref{main}},
we have
\begin{equation}
\label{seed}
\tilde{S}(n, k)
=\frac{1}{2}\sum_{j=k+1}^{n}(-1)^{n-j}\binom{j-1}{k-1}
\tilde{S}(n, j).
\end{equation}
Then,
we can show Theorem~{\rm\ref{main}}
by induction and by properties of the recurrence weight $R$.

Let $\Delta_{n, k}$ be the greatest common divisor of $\tilde{S}(n, j)$
for $j=k, k+1, \cdots, n$.
(See Lundell~\cite{Lundell_1978} for $\Delta_{n, k}$ where $k\leq 12$.)
Then, 
we have the following corollary of Theorem~{\rm\ref{main}}.
\begin{corollary}
\label{gcd}
For any odd integer $n\geq 3$ and any positive even integer $k$
such that $k<n$, 
the equation 
\[
 \Delta_{n, k}=\Delta_{n, k+1}
\]
holds.
In other words, $\tilde{S}(n, k)$ 
is a positive multiple of $\Delta_{n, k+1}$.
\end{corollary}
To prove Corollary~{\rm\ref{gcd}},
it suffices to show that
\[
 \nu_p(\tilde{S}(n, k))\geq \nu_p(\Delta_{n, k+1})
=\min\{\nu_p(\tilde{S}(n, j))\ \vert\ j=k+1, k+2, \cdots, n\}
\]
for any prime $p$, where $\nu_p(i)$ denotes the power of $p$
in the factorization of a non-zero integer $i$ into prime powers.
By \eqref{seed}, 
for any odd prime $p$, 
it is immediate that 
$\nu_p(\tilde{S}(n, k))\geq \nu_p(\Delta_{n, k+1})$.
For $p=2$, 
we invoke the theorem of Staudt~\cite{Staudt_1840} 
and Clausen~\cite{Clausen_1840}
which states that
\[
 B'_{2m+2}=\prod_{p:\text{prime}, \frac{2m+2}{p-1}\in\mathbb{Z}}p
\]
where $B'_s$ is the denominator of $B_s$
when the fraction is expressed in its lowest terms.
Then, 
by Theorem~{\rm\ref{main}},
we can estimate $\nu_2(\tilde{S}(n, k))$ so that 
$\nu_2(\tilde{S}(n, k))\geq \nu_2(\Delta_{n, k+1})$.
(See section~{\rm\ref{pf_of_gcd}} for details.)

It is known that the numbers $\Delta_{n, k}$ 
are important in algebraic topology.
(See for example \cite{BD_1991, CK_1988, Davis_1991, Davis_2008, 
DP_2007, Lundell_1974, Lundell_1978}.)
We have an application of Corollary~{\rm\ref{gcd}}
which concerns the unstable $K$-theory
of stunted complex projective spaces
and which is the motivation of this paper.
Let $\mathbb{C}P^n$ be the
complex projective space of complex lines in $\mathbb{C}^{n+1}$.\ 
We abbreviate the stunted complex projective space
$\mathbb{C}P^n/\mathbb{C}P^{k-1}$
by $\mathbb{C}P^n_k$.\ 
The set $U_n(X)=[X, U(n)]_*$ of pointed homotopy classes
of pointed maps from a topological space $X$ to 
the unitary group $U(n)$ of rank $n$ has
a group structure induced by the canonical multiplication of $U(n)$.\ 
(See Hamanaka and Kono~\cite{HK_2003}.)
Following \cite{HK_2003}, 
we call $U_n(X)$ the unstable $K^1$-group of $X$.
It is known that $U_n(X)$ is useful for studying
homotopy types of certain gauge groups. 
(See for example Hamanaka and Kono~\cite{HK_2006, HK_2007}.)
Using the result of \cite{HK_2003}, 
we can see that
$U_n(\mathbb{C}P^{n}_{k})$ is a cyclic group 
of order $\Delta_{n, k}$
and that we have a sequence of epimorphisms of groups
\begin{equation}
 \label{Useq}
 U_n(\mathbb{C}P^{n}_{n})
 \xrightarrow{q^*} U_n(\mathbb{C}P^{n}_{n-1})
\xrightarrow{q^*} \cdots
\xrightarrow{q^*} U_n(\mathbb{C}P^{n}_{3})
\xrightarrow{q^*} U_n(\mathbb{C}P^{n}_{2})
\xrightarrow{q^*} U_n(\mathbb{C}P^{n}_{1})=0
\end{equation}
where 
$q=q_{j, j+1}\colon \mathbb{C}P^{n}_{j}\to \mathbb{C}P^{n}_{j+1}$ 
is the natural projection
for $j=1, 2, \cdots, n-1$
and $q^*$ is the homomorphism induced by $q$.
(See Proposition~{\rm\ref{gcd_is_order}}.)
Then,
by Corollary~{\rm\ref{gcd}}, 
we have the following theorem.
\begin{theorem}
\label{unstable}
For any odd integer $n\geq 3$ and any positive even integer $k$
such that $k<n$, 
\[
q^*\colon U_n(\mathbb{C}P^{n}_{k+1})\to
U_n(\mathbb{C}P^{n}_{k})
\]
is an isomorphism of groups.
\end{theorem}

This paper is organized as follows.
In section~{\rm\ref{Def_of_R}}, 
we prepare for the recurrence weight $R$
stated above.
In section~{\rm\ref{pf_of_main}}, we prove Theorem~{\rm\ref{main}}.
In section~{\rm\ref{pf_of_gcd}}, we prove Corollary~{\rm\ref{gcd}}.
In section~{\rm\ref{pf_of_unstable}}, we prove Theorem~{\rm\ref{unstable}}
and give some corollaries and examples. 
Moreover,
in terms of the sequence \eqref{Useq},
we give a characterization of
the complex James number $b_k$
of Stiefel manifolds  
(see James~\cite{James_1958} and Atiyah~\cite{Atiyah_1961})
which is the same as the Atiyah-Todd number $M_k$ 
(see Atiyah and Todd~\cite{AT_1960})
by the result of Adams and Walker~\cite{AW_1965}.
(See Proposition~{\rm\ref{James}}.)

\section{Definition and some properties of the recurrence weight $R$}
\label{Def_of_R}
In this section, 
we define a function $R$,
which we call the recurrence weight,
and investigate its some properties.

For any positive integers $m$ and $j$,
let $M(m, j)$ be the central factorial number of the second kind
which is defined by
\[
 (e^z+e^{-z}-2)^j=\sum_{m=1}^{\infty}\frac{(2j)!}{(2m)!}M(m,j)z^{2m}
\]
for $z\in\mathbb{C}$.
Since 
\begin{equation}
\label{expansion}
e^z+e^{-z}-2
=2\sum_{m=1}^{\infty}\frac{z^{2m}}{(2m)!}
=2\left(\frac{z^2}{2!}+\frac{z^4}{4!}+\cdots\right),
\end{equation}
$M(m, j)$ is explicitly given by
\begin{align*}
 M(m, j)&=\frac{2^j(2m)!}{(2j)!}\sum_{(i_1, \cdots, i_j)\in E_{m, j}}
 \frac{1}{(2i_1)!(2i_2)!\cdots(2i_j)!} \\
 &=\frac{2^j}{(2j)!}\sum_{(I_1, \cdots, I_j)\in \tilde{E}_{m, j}}
 \binom{2I_2}{2I_1}\binom{2I_3}{2I_2}\cdots\binom{2I_j}{2I_{j-1}}
\end{align*}
where
\begin{align*}
 E_{m, j}&=\{(i_1, \cdots, i_j)\in \mathbb{Z}^j\ 
 \vert\ i_1>0, \cdots, i_j>0,\ i_1+\cdots +i_j=m\}, \\
 \tilde{E}_{m, j}&=\{(I_1, \cdots, I_j)\in \mathbb{Z}^j\ 
 \vert\ 0<I_1<I_2<\cdots <I_j=m\},
\end{align*}
and
\[
 \binom{2I_2}{2I_1}\binom{2I_3}{2I_2}\cdots\binom{2I_j}{2I_{j-1}}
\]
is to be interpreted as $1$ for $j=1$.

Let $R_E$ be the function defined on non-negative integers by
\begin{align*}
 R_E(m)&=\sum_{j=1}^m(-1)^j\frac{(2j)!}{2^{2j}}M(m, j) \\
       &=(2m)!\sum_{j=1}^m\left(-\frac{1}{2}\right)^j
\sum_{(i_1, \cdots, i_j)\in E_{m, j}}
 \frac{1}{(2i_1)!(2i_2)!\cdots(2i_j)!}  \\
       &=\sum_{j=1}^m \left(-\frac{1}{2}\right)^j
 \sum_{(I_1, \cdots, I_j)\in \tilde{E}_{m, j}}
 \binom{2I_2}{2I_1}\binom{2I_3}{2I_2}\cdots\binom{2I_j}{2I_{j-1}}
\in \mathbb{Q}  
\end{align*}
for any positive integer $m$
and by $R_E(0)=1$.
\begin{lemma}
\label{Lem_of_RE}
We have
 \[
 R_E(m)=\frac{2(2^{2m+2}-1)B_{2m+2}}{m+1}.
 \]
\end{lemma}
\begin{proof}
By \eqref{expansion},
we have
\[
\sum_{j=1}^{\infty}
\left(-\frac{1}{4}\right)^j(e^z+e^{-z}-2)^j
=\sum_{j=1}^{\infty}
\left(-\frac{1}{2}\right)^j
\left(\frac{z^2}{2!}+\frac{z^4}{4!}+\cdots\right)^j
\]
which is analytic on
\[
 D=\{z\in \mathbb{C}\ \vert\ 
\left\vert\frac{1}{4}(e^z+e^{-z}-2)\right\vert
=\left\vert\sinh^2 \frac{z}{2}\right\vert
<1\},
\] 
a neighborhood of $0$.
Here recall that 
\[
 \frac{d}{dz}(\tanh z)=1-\tanh^2z
\]
and hence,
\[
 \frac{d}{dz}\left(\tanh \frac{z}{2}\right)
=\frac{1}{2}\left(1-\tanh^2\frac{z}{2}\right).
\]
Also recall that
the Maclaurin series of $\tanh z$
is given by
\[
 \tanh z=\sum_{m=1}^{\infty}
\frac{2^{2m}(2^{2m}-1)B_{2m}}{(2m)!}z^{2m-1}
\]
and hence,
\begin{equation}
\label{Mac_of_tanh}
 \tanh \frac{z}{2}=\sum_{m=1}^{\infty}
\frac{2(2^{2m}-1)B_{2m}}{(2m)!}z^{2m-1}.
\end{equation}
Then, on one hand, we have
\begin{align*}
\sum_{j=1}^{\infty}
\left(-\frac{1}{4}\right)^j(e^z+e^{-z}-2)^j 
&=\frac{-\frac{1}{4}(e^z+e^{-z}-2)}{1+\frac{1}{4}(e^z+e^{-z}-2)} \\
&=-\frac{\sinh^2 \frac{z}{2}}{\cosh^2 \frac{z}{2}} \\
&=-\tanh^2 \frac{z}{2} \\
&=2\frac{d}{dz}\left(\tanh \frac{z}{2}\right)-1 \\
&=2\frac{d}{dz}\left(
\sum_{m=1}^{\infty}\frac{2(2^{2m}-1)B_{2m}}{(2m)!}z^{2m-1}
\right)-1 \\
&=\left(\sum_{m=1}^{\infty}
\frac{2^2(2^{2m}-1)B_{2m}(2m-1)}{(2m)!}z^{2m-2}\right)-1 \\
&=\sum_{m=1}^{\infty}
\frac{2^2(2^{2m+2}-1)B_{2m+2}(2m+1)}{(2m+2)!}z^{2m} \\
&=\sum_{m=1}^{\infty}
\frac{2(2^{2m+2}-1)B_{2m+2}}{(2m)!(m+1)}z^{2m}.
\end{align*}
On the other hand, we have
\begin{align*}
&\phantom{=}\sum_{j=1}^{\infty}
\left(-\frac{1}{2}\right)^j
\left(\frac{z^2}{2!}+\frac{z^4}{4!}+\cdots\right)^j \\
&=\sum_{m=1}^{\infty}
\left(\sum_{j=1}^m\left(-\frac{1}{2}\right)^j
\sum_{(i_1, \cdots, i_j)\in E_{m, j}}
 \frac{1}{(2i_1)!(2i_2)!\cdots(2i_j)!}\right)z^{2m} \\
&=\sum_{m=1}^{\infty}\frac{R_E(m)}{(2m)!}z^{2m}.
\end{align*}
Hence, we have
 \[
 R_E(m)=\frac{2(2^{2m+2}-1)B_{2m+2}}{m+1}
 \]
as desired.
\end{proof}

Similarly, let $R_O$ be the function defined on non-negative integers by
\begin{align*}
R_O(m)
&=(2m+1)!\sum_{j=0}^{m}\left(-\frac{1}{2}\right)^{j}
\!\!\!
\sum_{(i_0, i_1, \cdots, i_j)\in O_{m, j}}
 \frac{1}{(2i_0-1)!(2i_1)!\cdots(2i_j)!}  \\
&=\sum_{j=0}^m \left(-\frac{1}{2}\right)^j
\!\!\!
\sum_{(I_0, I_1, \cdots, I_j)\in \tilde{O}_{m, j}}
\!\!
\binom{2I_1-1}{2I_0-1}\binom{2I_2-1}{2I_1-1}\cdots\binom{2I_j-1}{2I_{j-1}-1}
\in \mathbb{Q}  
\end{align*}
for any non-negative integer $m$ where
\begin{align*}
 O_{m, j}&=\{(i_0, i_1, \cdots, i_j)\in \mathbb{Z}^{j+1}\ 
 \vert\ i_0>0, i_1>0, \cdots, i_j>0,\  \\
& \qquad \qquad \qquad \qquad \qquad \qquad \qquad \qquad 
i_0+i_1+\cdots +i_j=m+1\}, \\
 \tilde{O}_{m, j}&=\{(I_0, I_1, \cdots, I_j)\in \mathbb{Z}^{j+1}\ 
 \vert\ 0<I_0<I_1<\cdots <I_j=m+1\},
\end{align*}
and
\[
 \binom{2I_1-1}{2I_0-1}\binom{2I_2-1}{2I_1-1}\cdots\binom{2I_j-1}{2I_{j-1}-1}
\]
is to be interpreted as $1$ for $j=0$.
\begin{lemma}
\label{Lem_of_RO}
We have
 \[
 R_O(m)=\frac{2(2^{2m+2}-1)B_{2m+2}}{m+1}.
 \]
\end{lemma}
\begin{proof}
By
\[
 \sinh z
=\frac{1}{2}(e^{z}-e^{-z})
=\sum_{m=1}^{\infty}\frac{z^{2m-1}}{(2m-1)!}
=\frac{z^1}{1!}+\frac{z^3}{3!}+\cdots
\] 
and
\[
 \cosh z-1=\frac{1}{2}(e^{z}+e^{-z}-2)
=\sum_{m=1}^{\infty}\frac{z^{2m}}{(2m)!}
=\frac{z^2}{2!}+\frac{z^4}{4!}+\cdots
\]
for $z\in \mathbb{C}$,
we have
\begin{align*}
&\phantom{=}\sinh z\sum_{j=0}^{\infty}
 \left(-\frac{1}{2}\right)^j(\cosh z-1)^j \\
&=\left(\frac{z^1}{1!}+\frac{z^3}{3!}+\cdots\right)\sum_{j=0}^{\infty}
 \left(-\frac{1}{2}\right)^j
\left(\frac{z^2}{2!}+\frac{z^4}{4!}+\cdots\right)^j
\end{align*}
which is analytic on $D$.
Then, on one hand, we have
\begin{align*}
\sinh z\sum_{j=0}^{\infty}
\left(-\frac{1}{2}\right)^j(\cosh z-1)^j
&=\frac{\sinh z}{1+\frac{1}{2}(\cosh z-1)} \\
&=2\cdot\frac{\sinh z}{\cosh z+1} \\
&=2\cdot\frac{2\sinh \frac{z}{2}\cosh \frac{z}{2}}{2\cosh^2 \frac{z}{2}} \\
&=2\tanh \frac{z}{2} \\
&=\sum_{m=1}^{\infty}\frac{2^2(2^{2m}-1)B_{2m}}{(2m)!}z^{2m-1} \\
&=\sum_{m=0}^{\infty}
\frac{2(2^{2m+2}-1)B_{2m+2}}{(2m+1)!(m+1)}z^{2m+1}.
\end{align*}
On the other hand, we have
\begin{align*}
&\phantom{=}\left(\frac{z^1}{1!}+\frac{z^3}{3!}+\cdots\right)\sum_{j=0}^{\infty}
\left(-\frac{1}{2}\right)^j
\left(\frac{z^2}{2!}+\frac{z^4}{4!}+\cdots\right)^j \\
&=\sum_{m=0}^\infty\left(\sum_{j=0}^{m}\left(-\frac{1}{2}\right)^{j}
\sum_{(i_0, i_1, \cdots, i_j)\in O_{m, j}}
 \frac{1}{(2i_0-1)!(2i_1)!\cdots(2i_j)!}\right)z^{2m+1}  \\
&=\sum_{m=0}^{\infty}\frac{R_O(m)}{(2m+1)!}z^{2m+1}.
\end{align*}
Hence, we have
 \[
 R_O(m)=\frac{2(2^{2m+2}-1)B_{2m+2}}{m+1}
 \]
as desired.
\end{proof}

By Lemma~{\rm\ref{Lem_of_RE}} and Lemma~{\rm\ref{Lem_of_RO}}, we have
\begin{corollary}
\label{RERO}
The functions $R_E$ and $R_O$ coincide. In fact,
\[
 R_E(m)=R_O(m)=\frac{2(2^{2m+2}-1)B_{2m+2}}{m+1}.
\]
\end{corollary}
In the following, we abbreviate $R=R_E=R_O$.
As stated in the introduction,
we call $R$ 
{\em the recurrence weight\/}.
Here we give two fundamental recurrence formulae for $R$,
of which
one is derived from the definition of $R_E$
and the other from the definition of $R_O$.
\begin{proposition}
\label{Rec_of_RE}
The recurrence weight $R=R_E$ satisfies the relation
\[
 R(m+1)=-\frac{1}{2}\sum_{h=0}^{m}R(h)\binom{2m+2}{2h}
\]
for any non-negative integer $m$.
\end{proposition}
\begin{proof}
We have
\begin{align*}
 &\phantom{=}-\frac{1}{2}\sum_{h=0}^{m}R(h)\binom{2m+2}{2h} \\
 &=-\frac{1}{2}-\frac{1}{2}\sum_{h=1}^{m}R(h)\binom{2m+2}{2h} \\
 &=-\frac{1}{2}+\sum_{h=1}^{m}
 \sum_{j=1}^h
 \left(-\frac{1}{2}\right)^{j+1}
 \sum_{(I_1, \cdots, I_j)\in \tilde{E}_{h, j}}
 \binom{2I_2}{2I_1}
\cdots\binom{2I_j}{2I_{j-1}}
\binom{2m+2}{2h} \\
 &=-\frac{1}{2}+\sum_{j=1}^{m}
\left(-\frac{1}{2}\right)^{j+1}
\sum_{h=j}^{m}
\ 
\sum_{(I_1, \cdots, I_j)\in \tilde{E}_{h, j}}
\binom{2I_2}{2I_1}
\cdots\binom{2I_j}{2I_{j-1}}
\binom{2m+2}{2I_j}.
\end{align*}
Here note that $\tilde{E}_{m+1, j+1}$ is the disjoint union of
\begin{align*}
 \tilde{E}_{m+1, j+1}^{(h)}&=\{(I_1, \cdots, I_j, I_{j+1})
\in \tilde{E}_{m+1, j+1}\ 
 \vert\ I_j=h\} \\
&=\{(I_1, \cdots, I_j, m+1)\in \mathbb{Z}^{j+1}\ 
 \vert\ (I_1, \cdots, I_j)\in \tilde{E}_{h, j}\}
\end{align*}
for $h=j, j+1, \cdots, m$.
Then, we have
\begin{align*}
 &\phantom{=}-\frac{1}{2}\sum_{h=0}^{m}R(h)\binom{2m+2}{2h} \\
 &=-\frac{1}{2}+\sum_{j=1}^{m}
\left(-\frac{1}{2}\right)^{j+1}
\sum_{h=j}^{m}
\sum_{(I_1, \cdots, I_j, I_{j+1})\in \tilde{E}_{m+1, j+1}^{(h)}}
\binom{2I_2}{2I_1}
\cdots
\binom{2I_{j+1}}{2I_j} \\
 &=-\frac{1}{2}+\sum_{j=1}^{m}
\left(-\frac{1}{2}\right)^{j+1}
\sum_{(I_1, \cdots, I_j, I_{j+1})\in \tilde{E}_{m+1, j+1}}
\binom{2I_2}{2I_1}
\cdots
\binom{2I_{j+1}}{2I_j} \\
 &=-\frac{1}{2}+\sum_{j=2}^{m+1}
\left(-\frac{1}{2}\right)^{j}
\sum_{(I_1, \cdots, I_{j})\in \tilde{E}_{m+1, j}}
\binom{2I_2}{2I_1}
\cdots\binom{2I_j}{2I_{j-1}} \\
 &=\sum_{j=1}^{m+1}
\left(-\frac{1}{2}\right)^{j}
\sum_{(I_1, \cdots, I_{j})\in \tilde{E}_{m+1, j}}
\binom{2I_2}{2I_1}
\cdots\binom{2I_j}{2I_{j-1}} \\
 &=R(m+1)
\end{align*}
as desired.
\end{proof}
\begin{proposition}
\label{Rec_of_RO}
The recurrence weight $R=R_O$ satisfies the relation
\[
 R(m+1)=1-\frac{1}{2}\sum_{h=0}^{m}R(h)\binom{2m+3}{2h+1}
\]
for any non-negative integer $m$.
\end{proposition}
\begin{proof}
We have
\begin{align*}
 &\phantom{=}1-\frac{1}{2}\sum_{h=0}^{m}R(h)\binom{2m+3}{2h+1} \\
 &=1+\sum_{h=0}^{m}
 \sum_{j=0}^h \left(-\frac{1}{2}\right)^{j+1}
 \sum_{(I_0, I_1, \cdots, I_j)\in \tilde{O}_{h, j}}
\!\!\!
 \binom{2I_1-1}{2I_0-1}
\cdots\binom{2I_j-1}{2I_{j-1}-1}
 \binom{2m+3}{2h+1} \\
 &=1+\sum_{j=0}^{m}\left(-\frac{1}{2}\right)^{j+1}
 \sum_{h=j}^{m} 
 \ 
 \sum_{(I_0, I_1, \cdots, I_j)\in \tilde{O}_{h, j}}
\!\!\!
 \binom{2I_1-1}{2I_0-1}
\cdots\binom{2I_j-1}{2I_{j-1}-1}
 \binom{2m+3}{2I_j-1}.
\end{align*}
Here note that $\tilde{O}_{m+1, j+1}$ is the disjoint union of
\begin{align*}
 \tilde{O}_{m+1, j+1}^{(h)}&=\{(I_0, I_1, \cdots, I_j, I_{j+1})
\in \tilde{O}_{m+1, j+1}\ 
 \vert\ I_j=h+1\} \\
&=\{(I_0, I_1, \cdots, I_j, m+2)\in \mathbb{Z}^{j+2}\ 
 \vert\ (I_0, I_1, \cdots, I_j)\in \tilde{O}_{h, j}\}
\end{align*}
for $h=j, j+1, \cdots, m$.
Then, we have
\begin{align*}
 &\phantom{=}1-\frac{1}{2}\sum_{h=0}^{m}R(h)\binom{2m+3}{2h+1} \\
 &=1+\sum_{j=0}^{m}\left(-\frac{1}{2}\right)^{j+1}
 \sum_{h=j}^{m} 
 \ 
 \sum_{(I_0, I_1, \cdots, I_j, I_{j+1})\in \tilde{O}_{m+1, j+1}^{(h)}}
 \binom{2I_1-1}{2I_0-1}
 \cdots
 \binom{2I_{j+1}-1}{2I_j-1} \\
 &=1+\sum_{j=0}^{m}\left(-\frac{1}{2}\right)^{j+1}
 \sum_{(I_0, I_1, \cdots, I_j, I_{j+1})\in \tilde{O}_{m+1, j+1}}
 \binom{2I_1-1}{2I_0-1}
 \cdots
 \binom{2I_{j+1}-1}{2I_j-1} \\
 &=1+\sum_{j=1}^{m+1}\left(-\frac{1}{2}\right)^{j}
 \sum_{(I_0, I_1, \cdots, I_j)\in \tilde{O}_{m+1, j}}
 \binom{2I_1-1}{2I_0-1}
 \cdots\binom{2I_j-1}{2I_{j-1}-1} \\
 &=\sum_{j=0}^{m+1}\left(-\frac{1}{2}\right)^{j}
 \sum_{(I_0, I_1, \cdots, I_j)\in \tilde{O}_{m+1, j}}
 \binom{2I_1-1}{2I_0-1}
 \cdots\binom{2I_j-1}{2I_{j-1}-1} \\
 &=R(m+1)
\end{align*}
as desired.
\end{proof}
Thus, we can compute the recurrence weight $R$ recursively as
\begin{align*}
  R(1)&=-\frac{1}{2}R(0)\binom{2}{0}=-\frac{1}{2},\ \\
  R(2)&=-\frac{1}{2}
\left(R(0)\binom{4}{0}+R(1)\binom{4}{2}\right)=1,\ \\
  R(3)&=-\frac{1}{2}
\left(R(0)\binom{6}{0}+R(1)\binom{6}{2}+R(2)\binom{6}{4}\right)
=-\frac{17}{4},\ \\
  R(4)&=-\frac{1}{2}
\left(R(0)\binom{8}{0}+R(1)\binom{8}{2}+R(2)\binom{8}{4}
+R(3)\binom{8}{6}\right)
=31,\ \\
&\cdots
\end{align*}
or as
\begin{align*}
  R(1)&=1-\frac{1}{2}R(0)\binom{3}{1}=-\frac{1}{2},\ \\
  R(2)&=1-\frac{1}{2}
\left(R(0)\binom{5}{1}+R(1)\binom{5}{3}\right)=1,\ \\
  R(3)&=1-\frac{1}{2}
\left(R(0)\binom{7}{1}+R(1)\binom{7}{3}+R(2)\binom{7}{5}\right)
=-\frac{17}{4},\ \\
  R(4)&=1-\frac{1}{2}
\left(R(0)\binom{9}{1}+R(1)\binom{9}{3}+R(2)\binom{9}{5}
+R(3)\binom{9}{7}\right)
=31,\ \\
&\cdots .
\end{align*}
In computing $R$ recursively by hand,
it seems to be easier to use 
Proposition~{\rm\ref{Rec_of_RE}}
than to use Proposition~{\rm\ref{Rec_of_RO}},
though we use the latter
to prove Theorem~{\rm\ref{main}} later.

\begin{corollary}
\label{Rec_of_Ber}
We have
\[
B_{2m+2}
=-\frac{m+1}{2(2^{2m+2}-1)}\sum_{h=1}^{m}
\frac{(2^{2h}-1)B_{2h}}{h}
\binom{2m}{2h-2}
\]
and
\[
B_{2m+2}
=\frac{m+1}{2(2^{2m+2}-1)}\left(1-\sum_{h=1}^{m}
\frac{(2^{2h}-1)B_{2h}}{h}
\binom{2m+1}{2h-1}\right)
\]
for any positive integer $m$.
\end{corollary}
\begin{proof}
In fact, by Corollary~{\rm\ref{RERO}} 
and Proposition~{\rm\ref{Rec_of_RE}}, we have
\begin{align*}
 \frac{2(2^{2m+2}-1)B_{2m+2}}{m+1}
&=-\frac{1}{2}\sum_{h=0}^{m-1}
 \frac{2(2^{2h+2}-1)B_{2h+2}}{h+1}
 \binom{2m}{2h} \\
&=-\sum_{h=1}^{m}
 \frac{(2^{2h}-1)B_{2h}}{h}
 \binom{2m}{2h-2}
\end{align*}
and hence,
\[
B_{2m+2}
=-\frac{m+1}{2(2^{2m+2}-1)}\sum_{h=1}^{m}
\frac{(2^{2h}-1)B_{2h}}{h}
\binom{2m}{2h-2}.
\]
Thus,
the former follows.
Similarly,
the latter follows 
from Corollary~{\rm\ref{RERO}} and Proposition~{\rm\ref{Rec_of_RO}}.
\end{proof}
The author cannot find these recurrence formulae in the literature.
We can deduce other recurrence formulae for the recurrence weight $R$
and then we can transform them to ones 
for Bernoulli numbers by Corollary~{\rm\ref{RERO}}.
In the following, 
for example,
we give two other 
simple recurrence formulae for the recurrence weight $R$.
\begin{proposition}
\label{Rec_of_RERO}
The recurrence weight $R=R_E=R_O$ satisfies the relation
\[
 R(m)=\sum_{h=0}^{m}R(h)\binom{2m+1}{2h}
\]
and hence, 
it satisfies the relation
\[
 R(m)=-\frac{1}{2m}\sum_{h=0}^{m-1}R(h)\binom{2m+1}{2h}
\]
for any positive integer $m$.
\end{proposition}
\begin{proof}
Since we have
\begin{equation}
\label{binom}
\binom{b}{c}\binom{a}{b}=\binom{a}{c}\binom{a-c}{a-b}
\end{equation}
for any non-negative integers $a, b, c$ such that $c\leq b\leq a$,
we can see that
for $(I_0, I_1, \cdots, I_j)\in \tilde{O}_{m, j}$,
\begin{align*}
&\phantom{=}
\binom{2I_1-1}{2I_0-1}
\cdots
\binom{2I_{j-2}-1}{2I_{j-3}-1}
\binom{2I_{j-1}-1}{2I_{j-2}-1}
\binom{2I_j-1}{2I_{j-1}-1} \\
&=\binom{2I_1-1}{2I_0-1}
\cdots
\binom{2I_{j-2}-1}{2I_{j-3}-1}
\binom{2I_{j}-1}{2I_{j-2}-1}\binom{2(I_j-I_{j-2})}{2(I_{j}-I_{j-1})} \\
&=\binom{2I_1-1}{2I_0-1}
\cdots
\binom{2I_{j}-1}{2I_{j-3}-1}
\binom{2(I_{j}-I_{j-3})}{2(I_{j}-I_{j-2})}
\binom{2(I_{j}-I_{j-2})}{2(I_{j}-I_{j-1})} \\
&=\cdots \\
&=\binom{2I_{j}-1}{2I_{0}-1}
\binom{2(I_{j}-I_{0})}{2(I_{j}-I_{1})}
\binom{2(I_{j}-I_{1})}{2(I_{j}-I_{2})}
\cdots
\binom{2(I_{j}-I_{j-2})}{2(I_{j}-I_{j-1})} \\
&=\binom{2m+1}{2(I_{j}-I_{0})}
\binom{2(I_{j}-I_{0})}{2(I_{j}-I_{1})}
\binom{2(I_{j}-I_{1})}{2(I_{j}-I_{2})}
\cdots
\binom{2(I_{j}-I_{j-2})}{2(I_{j}-I_{j-1})} \\
&=\binom{2m+1}{2(I_{j}-I_{0})}
\binom{2(I_{j}-I_{j-2})}{2(I_{j}-I_{j-1})}
\cdots
\binom{2(I_{j}-I_{1})}{2(I_{j}-I_{2})}
\binom{2(I_{j}-I_{0})}{2(I_{j}-I_{1})}.
\end{align*}
Thus, we have
\begin{align*}
&\phantom{=}
\sum_{j=1}^m \left(-\frac{1}{2}\right)^j
\sum_{(I_0, I_1, \cdots, I_j)\in \tilde{O}_{m, j}}
\binom{2I_1-1}{2I_0-1}
\cdots
\binom{2I_j-1}{2I_{j-1}-1} \\
&=\sum_{j=1}^m \left(-\frac{1}{2}\right)^j
\sum_{(I_0, I_1, \cdots, I_j)\in \tilde{O}_{m, j}}
\binom{2m+1}{2(I_{j}-I_{0})}
\binom{2(I_{j}-I_{j-2})}{2(I_{j}-I_{j-1})}
\cdots
\binom{2(I_{j}-I_{0})}{2(I_{j}-I_{1})}.
\end{align*}
Here let
\[
 (I_{j}-I_{j-1}, I_{j}-I_{j-2}, \cdots, I_{j}-I_{1}, I_{j}-I_{0})
=(I'_1, I'_2, \cdots, I'_j).
\]
Note that as $(I_0, I_1, \cdots, I_j)$ runs over $\tilde{O}_{m, j}$
with $1\leq j\leq m$,
$(I'_1, I'_2, \cdots, I'_j)$
bijectively runs over $\tilde{E}_{h, j}$
with $1\leq j\leq h\leq m$.
Then, 
it follows that
the above sum is equal to
\begin{align*}
&\phantom{=}\sum_{h=1}^{m}
 \binom{2m+1}{2h}
 \left(
 \sum_{j=1}^{h}
 \left(-\frac{1}{2}\right)^j
 \sum_{(I'_1, I'_2 \cdots, I'_j)\in \tilde{E}_{h, j}}
 \binom{2I'_2}{2I'_1}
 \cdots
 \binom{2I'_{j}}{2I'_{j-1}}
 \right) \\
&=\sum_{h=1}^{m}R(h)\binom{2m+1}{2h}.
\end{align*}
Thus we have
\begin{align*}
 R(m)&=\sum_{j=0}^m \left(-\frac{1}{2}\right)^j
 \sum_{(I_0, I_1, \cdots, I_j)\in \tilde{O}_{m, j}}
\binom{2I_1-1}{2I_0-1}
\cdots\binom{2I_j-1}{2I_{j-1}-1} \\
&=1+\sum_{j=1}^m \left(-\frac{1}{2}\right)^j
\sum_{(I_0, I_1, \cdots, I_j)\in \tilde{O}_{m, j}}
\binom{2I_1-1}{2I_0-1}
\cdots\binom{2I_j-1}{2I_{j-1}-1} \\
&=1+\sum_{h=1}^{m}R(h)\binom{2m+1}{2h} \\
&=\sum_{h=0}^{m}R(h)\binom{2m+1}{2h}
\end{align*}
as desired.
\end{proof}
\begin{corollary}
\label{Rec_of_Ber2}
We have
\[
B_{2m+2}=
-\frac{m+1}{2m(2^{2m+2}-1)}
\sum_{h=1}^{m}\frac{(2^{2h}-1)B_{2h}}{h}\binom{2m+1}{2h-2}
\]
for any positive integer $m$.
\end{corollary}
\begin{proof}
This follows from Corollary~{\rm\ref{RERO}}
and the previous proposition.
\end{proof}
\begin{proposition}
\label{Rec_of_comparison}
The recurrence weight $R$ satisfies the relation
\[
 \sum_{h=0}^{m}2^{2h}R(h)\binom{2m+1}{2h+1}=1
\]
for any non-negative integer $m$.
\end{proposition}
\begin{proof}
By Lemma~{\rm\ref{Lem_of_RO}} and its proof,
\begin{align*}
&\phantom{=}\sinh z\sum_{j=0}^{\infty}
 \left(-\frac{1}{2}\right)^j(\cosh z-1)^j \\
&=\left(\frac{z^1}{1!}+\frac{z^3}{3!}+\cdots\right)\sum_{j=0}^{\infty}
 \left(-\frac{1}{2}\right)^j
 \left(\frac{z^2}{2!}+\frac{z^4}{4!}+\cdots\right)^j
\end{align*}
has an analytic continuation
\[
 2\tanh \frac{z}{2}=\sum_{m=0}^{\infty}\frac{R(m)}{(2m+1)!}z^{2m+1}
\]
on
\[
 \tilde{D}=\{z\in\mathbb{C}\ \vert\ \lvert z\rvert<\pi\}.
\]
Then,
by 
\begin{equation}
\label{tanhcoshsinh}
\tanh \frac{z}{2} \cosh \frac{z}{2}=\sinh \frac{z}{2}
\end{equation}
on $\tilde{D}$, we have
\[
 \left(\frac{1}{2}\sum_{m=0}^{\infty}\frac{R(m)}{(2m+1)!}z^{2m+1}\right)
\left(\sum_{m=0}^{\infty}\frac{z^{2m}}{2^{2m}(2m)!}\right)
=\sum_{m=1}^{\infty}\frac{z^{2m-1}}{2^{2m-1}(2m-1)!}
\]
and hence
\[
 \frac{1}{2}\sum_{h=0}^{m}\frac{R(h)}{(2h+1)!}
\cdot\frac{1}{2^{2m-2h}(2m-2h)!}
=\frac{1}{2^{2m+1}(2m+1)!}.
\]
It follows that
\[
\sum_{h=0}^{m}
2^{2h}R(h)\frac{(2m+1)!}{(2h+1)!(2m-2h)!}
=\sum_{h=0}^{m}
2^{2h}R(h)\binom{2m+1}{2h+1}
=1
\]
as desired.
\end{proof}
The recurrence formula of Bernoulli numbers
which is derived from
Corollary~{\rm\ref{RERO}} 
and this proposition
is of course merely the one derived from
the Maclaurin series \eqref{Mac_of_tanh}
of $\tanh \frac{z}{2}$, 
those of
$\cosh \frac{z}{2}$ and $\sinh \frac{z}{2}$,
and the relation \eqref{tanhcoshsinh}.

\section{Proof of Theorem~{\rm\ref{main}}}
\label{pf_of_main}
In this section,
we prove Theorem~{\rm\ref{main}}.

Recall that 
Stirling numbers of the second kind
$S(n, k)$ satisfy the conditions that
\[
 S(n, 1)=1,\quad S(n, k)=0
\]
for any positive integers $n$ and $k$
such that $n<k$, 
and also satisfy
the recurrence formula
\begin{equation}
\label{Rec_of_Stirling}
S(n+1, k+1)=S(n, k)+(k+1)S(n, k+1)
\end{equation}
for any positive integers $n$ and $k$
such that $k\leq n$.
It is known that Stirling numbers 
of the second kind $S(n, k)$ satisfy 
many combinatorial identities
which are derived from \eqref{Rec_of_Stirling}. 
(See Quaintance and Gould~\cite{QG_2016}.)
Here, 
we use the identity 
\[
 \tilde{S}(n, k)
=\sum_{j=k}^{n}(-1)^{n-j}\binom{j-1}{k-1}
\tilde{S}(n, j)
\]
for any positive integers $n$ and $k$ such that $k\leq n$.
(See Quaintance and Gould~\cite[(9.18)]{QG_2016}.)
In particular,
for any odd integer $n\geq 3$ and any positive even integer $k$
such that $k<n$, 
we have
\[
 \tilde{S}(n, k)
=-\tilde{S}(n, k)+\sum_{j=k+1}^{n}(-1)^{n-j}\binom{j-1}{k-1}
\tilde{S}(n, j)
 \]
and hence
\begin{equation}
\label{seed_again}
 \tilde{S}(n, k)
=\frac{1}{2}\sum_{j=k+1}^{n}(-1)^{n-j}\binom{j-1}{k-1}
\tilde{S}(n, j).
\end{equation}

Now we prove Theorem~{\rm\ref{main}}. 
In the following,
we prove the equation 
\[
 \tilde{S}(k+2m+1, k)
=\frac{1}{2}\sum_{h=0}^{m}R(h)
\binom{k+2h}{k-1}\tilde{S}(k+2m+1, k+2h+1)
\]
for any non-negative integer $m$
and any positive even integer $k$,
which is a paraphrase of Theorem~{\rm\ref{main}},
by induction on $m$.

For $m=0$, the assertion is obvious by \eqref{seed_again}.
Suppose that the assertion is proved 
for $m\leq a$
where $a$ is a non-negative integer.
In the following,
we abbreviate $\tilde{S}(k+2(a+1)+1, j)=S'(j)$.
Then, 
by \eqref{seed_again},
we have
\begin{align*}
&\phantom{=}S'(k)
=\frac{1}{2}\sum_{j=k+1}^{k+2(a+1)+1}(-1)^{1-j}\binom{j-1}{k-1}S'(j) \\
&=\frac{1}{2}\binom{k}{k-1}S'(k+1)
+\frac{1}{2}\sum_{h=1}^{a+1}\binom{k+2h}{k-1}S'(k+2h+1) \\
&\qquad \qquad \qquad \qquad \qquad \qquad \qquad
-\frac{1}{2}\sum_{i=1}^{a+1}\binom{k+2i-1}{k-1}S'(k+2i).
\end{align*}
Here, 
we have
\begin{align*}
&\phantom{=}
\sum_{i=1}^{a+1}\binom{k+2i-1}{k-1}S'(k+2i) \\
&=\sum_{i=1}^{a+1}\binom{k+2i-1}{k-1}
\left(
\frac{1}{2}\sum_{\ell=0}^{a+1-i}R(\ell)
\binom{k+2i+2\ell}{k+2i-1}S'(k+2i+2\ell+1)
\right) \\
&\qquad \qquad \qquad \qquad \qquad \qquad 
\qquad \qquad \qquad \qquad
\text{(by the induction hypothesis)} \\
&=\frac{1}{2}
\sum_{i=1}^{a+1}
\sum_{\ell=0}^{a+1-i}
R(\ell)
\binom{k+2i-1}{k-1}
\binom{k+2i+2\ell}{k+2i-1}S'(k+2i+2\ell+1)  \\
&=\frac{1}{2}
\sum_{i=1}^{a+1}
\sum_{\ell=0}^{a+1-i}
R(\ell)
\binom{2i+2\ell+1}{2\ell+1}
\binom{k+2i+2\ell}{k-1}S'(k+2i+2\ell+1) 
\ \  \text{(by \eqref{binom})} \\
&=\frac{1}{2}
\sum_{h=1}^{a+1}
 \sum_{\ell=0}^{h-1}
R(\ell)
\binom{2h+1}{2\ell+1}
\binom{k+2h}{k-1}S'(k+2h+1)
\quad  \text{(by letting $i+\ell=h$)}.
\end{align*}
Hence, we have
\begin{align*}
S'(k)&=\frac{1}{2}\binom{k}{k-1}S'(k+1)
+\frac{1}{2}\sum_{h=1}^{a+1}\binom{k+2h}{k-1}S'(k+2h+1) \\
&\qquad \qquad \qquad -\frac{1}{2^2}
\sum_{h=1}^{a+1}
\sum_{\ell=0}^{h-1}
R(\ell)
\binom{2h+1}{2\ell+1}
\binom{k+2h}{k-1}S'(k+2h+1) \\
&=\frac{1}{2}\binom{k}{k-1}S'(k+1) \\
&\qquad \qquad +\frac{1}{2}
\sum_{h=1}^{a+1}
\left(
1-\frac{1}{2}\sum_{\ell=0}^{h-1}
R(\ell)
\binom{2h+1}{2\ell+1}
\right)
\binom{k+2h}{k-1}S'(k+2h+1) \\
&=\frac{1}{2}\binom{k}{k-1}S'(k+1)
+\frac{1}{2}
\sum_{h=1}^{a+1}
R(h)
\binom{k+2h}{k-1}S'(k+2h+1) \\
&\qquad \qquad \qquad \qquad \qquad \qquad \qquad \qquad 
\qquad \qquad \qquad 
\text{(by Proposition~{\rm\ref{Rec_of_RO}})} \\
&=\frac{1}{2}
\sum_{h=0}^{a+1}
R(h)
\binom{k+2h}{k-1}S'(k+2h+1),
\end{align*}
that is,
\[
 \tilde{S}(k+2(a+1)+1, k)=\frac{1}{2}
\sum_{h=0}^{a+1}
R(h)
\binom{k+2h}{k-1}\tilde{S}(k+2(a+1)+1, k+2h+1).
\]
Thus, 
by induction on $m$,
Theorem~{\rm\ref{main}} follows.

\section{Proof of Corollary~{\rm\ref{gcd}}}
\label{pf_of_gcd}
In this section,
we prove Corollary~{\rm\ref{gcd}}.

Let 
\[
 t=\nu_p(\Delta_{n, k+1})
=\min\{\nu_p(\tilde{S}(n, j))\ \vert\ j=k+1, k+2, \cdots, n\}.
\]
As stated in the introduction,
to prove Corollary~{\rm\ref{gcd}},
it suffices to show that
$\nu_p(\tilde{S}(n, k))\geq t$
for any odd integer $n\geq 3$,
any positive even integer $k$
such that $k<n$,
and any prime $p$.
In other words, 
it suffices to show that
$\tilde{S}(n, k)$ is divisible by $p^t$.

First let $p$ be an odd prime. 
Consider the equation \eqref{seed_again}:
\[
 \tilde{S}(n, k)
=\frac{1}{2}\sum_{j=k+1}^{n}(-1)^{n-j}\binom{j-1}{k-1}
\tilde{S}(n, j).
\]
Then, 
since the right hand side is 
of course an integer and
is divisible by $p^t$,
so is the left hand side.

In the rest of this section, 
let $p=2$.
Consider the equation of Theorem~{\rm\ref{main}}:
\[
 \tilde{S}(n, k)
=\frac{1}{2}\sum_{h=0}^{(n-k-1)/2}R(h)
\binom{k+2h}{k-1}\tilde{S}(n, k+2h+1).
\]
Since $\tilde{S}(n, k+2h+1)$ is divisible by $2^t$,
it suffices to show that
\[
 \nu_2\left(\frac{1}{2}R(h)\binom{k+2h}{k-1}\right)\geq 0
\]
for $h=0, 1, \cdots, (n-k-1)/2$
where $\nu_2(q_1/q_2)=\nu_2(q_1)-\nu_2(q_2)$
for a non-zero rational number 
$q_1/q_2\in\mathbb{Q},\ 0\ne q_1\in \mathbb{Z},\ 0\ne q_2\in \mathbb{Z}$.
As stated in the introduction,
we invoke the theorem of Staudt~\cite{Staudt_1840} 
and Clausen~\cite{Clausen_1840}
which states that
\[
 B'_{2m+2}=\prod_{p:\text{prime}, \frac{2m+2}{p-1}\in\mathbb{Z}}p.
\]
Hence $\nu_2(B_{2h+2})=-1$ and
\[
 \nu_2(R(h))=\nu_2\left(\frac{2(2^{2h+2}-1)B_{2h+2}}{h+1}\right)
=\nu_2\left(\frac{1}{h+1}\right).
\]
It follows that
\begin{align*}
\nu_2\left(\frac{1}{2}R(h)\binom{k+2h}{k-1}\right)
&=\nu_2\left(\frac{1}{2}\cdot \frac{1}{h+1}
\cdot\frac{(k+2h)!}{(k-1)!(2h+1)!}\right) \\
&=\nu_2\left(\frac{(k+2h+1)!}{(k-1)!(2h+2)!}\cdot\frac{1}{k+2h+1}\right) \\
&=\nu_2\left(\binom{k+2h+1}{k-1}\right) \\
&\geq 0,
\end{align*}
which completes the proof.

\section{Proof of Theorem~{\rm\ref{unstable}}}
\label{pf_of_unstable}
In this section, 
we recall some facts on the unstable $K$-theory
from Hamanaka and Kono~\cite{HK_2003} and then,
we prove Theorem~{\rm\ref{unstable}}.
Also we give some examples of Theorem~{\rm\ref{unstable}}
and a characterization of $b_k=M_k$
in terms of the sequence \eqref{Useq}.

As stated in the introduction,
let $\mathbb{C}P^n$ be the complex projective space
and $\mathbb{C}P^n_k=\mathbb{C}P^n/\mathbb{C}P^{k-1}.$
Let $U(n)$ be the unitary group and
$U_n(X)=[X, U(n)]_*$.
According to Hamanaka and Kono~\cite{HK_2003},\ 
for a CW-complex $X$ of dimension
less than or equal to $2n$,\ 
the group $U_n(X)$ fits into an exact sequence
\begin{equation}
\label{HKseq}
 \tilde{K}^0(X)\xrightarrow{\Theta} H^{2n}(X)\to U_n(X)\to 
\tilde{K}^{-1}(X)\to 0 
\end{equation}
of groups. 
Here $H^{\ell}(X)$ and $\tilde{K}^{\ell}(X)$ are 
the $\ell$th integral cohomology group of $X$
and the $\ell$th reduced complex $K$ cohomology group
of $X$ respectively,
and the homomorphism $\Theta$ is given as follows.

Let $T$ be a maximal torus of $U(n)$,
$i\colon T\to U(n)$ the inclusion,
$BT$ and $BU(n)$ the classifying spaces 
of $T$ and $U(n)$ respectively, 
and $Bi\colon BT\to BU(n)$ the map
induced by $i$.
Recall that
\[
 H^*(BT)=\mathbb{Z}[t_1, t_2, \cdots, t_n],
\]
the polynomial algebra generated by
$t_1, t_2, \cdots, t_n\in H^2(BT)$.
Let $W(U(n))$ be the Weyl group of $U(n)$
which is isomorphic to
the symmetric group of $n$ letters
and which acts on $H^*(BT)$
by permuting $t_1, t_2, \cdots, t_n$.
Let $H^*(BT)^{W(U(n))}$ be the algebra
which consists of
the polynomials in $H^*(BT)$
invariant under the action of $W(U(n))$.
Then it is well known that
\[
 (Bi)^*\colon H^*(BU(n))
=\mathbb{Z}[c_1, c_2, \cdots, c_n]
\xrightarrow{\cong}H^*(BT)^{W(U(n))}
\]
is an isomorphism of algebras
where $c_r\in H^{2r}(BU(n))$ is the $r$th universal Chern class
which corresponds to the $r$th elementary symmetric function
of $t_1, t_2, \cdots, t_n$
for $r\leq n$.
Let $s_r\in H^{2r}(BU(n))$ be the element
which corresponds to 
$t_1^r+t_2^r+\cdots +t_n^r\in H^{2r}(BT)^{W(U(n))}$
through $(Bi)^*$ and 
take the polynomial function $f_r$ of $r$ variables
such that $s_r=f_r(c_1, c_2, \cdots, c_r)$
for $r\leq n$.

Let $U$ be the infinite unitary group,
$BU$ the classifying space of $U$,
$i_n\colon U(n)\to U$ the inclusion,
and $Bi_n\colon BU(n)\to BU$ the map induced by $i_n$.
Recall that
\[
 H^*(BU)=\mathbb{Z}[c_1, c_2, \cdots, c_n, c_{n+1}, \cdots]
\]
where $c_r\in H^{2r}(BU)$ is the $r$th universal Chern class
which corresponds to $c_r\in H^{2r}(BU(n))$ if $r\leq n$
through the homomorphism 
$(Bi_n)^*\colon H^{2r}(BU)\to H^{2r}(BU(n))$ induced by $Bi_n$.
Denote the element $f_r(c_1, c_2, \cdots, c_r)$
for $c_1$, $c_2$, $\cdots$, $c_r\in H^*(BU)$
by the same letter $s_r\in H^{2r}(BU)$ as 
$s_r\in H^{2r}(BU(n))$
for $r\leq n$.

For a CW-complex $X$ 
and for $\theta\in \tilde{K}^0(X)=[X, BU]_*$,
let $s_r(\theta)=\theta^*(s_r)$,
the image of $s_r\in H^{2r}(BU)$
through the homomorphism 
$\theta^*\colon H^{2r}(BU)\to H^{2r}(X)$
induced by $\theta$.
Then, 
$\Theta(\theta)$ is given by
\[
 \Theta(\theta)=(-1)^ns_n(\theta)\in H^{2n}(X).
\]
Moreover, 
recall that for a CW-complex $X$ of dimension
less than or equal to $2n$ 
and for $\theta\in \tilde{K}^0(X)=[X, BU]_*$,
\[
 \ch(\theta)=\iota(s_1(\theta))
+\frac{\iota(s_2(\theta))}{2!}
+\cdots 
+\frac{\iota(s_n(\theta))}{n!}
=\sum_{r=1}^n\frac{\iota(s_r(\theta))}{r!}
\]
where
\[
 \ch\colon K(X)
=\mathbb{Z}\oplus \tilde{K}^{0}(X)\to \bigoplus_{r=0}^{n}H^{2r}(X; \mathbb{Q})
\]
is the Chern character 
and
$\iota\colon H^*(BU)\to H^*(BU ; \mathbb{Q})$
is induced by the inclusion $\mathbb{Z}\to \mathbb{Q}$.
Thus, 
if $H^{2n}(X)$ is free, 
then
$\iota\colon H^{2n}(X)\to H^{2n}(X ; \mathbb{Q})$
is monomorphic
and hence 
$\Theta(\theta)$ is determined by
\[
 \iota(\Theta(\theta))=(-1)^nn!\ch_{n}(\theta)
\in H^{2n}(X ; \mathbb{Q})
\]
where $\ch_n$ is the component of $\ch$
in $H^{2n}(X; \mathbb{Q})$.

Now we consider the case $X=\mathbb{C}P^n_k$.
(See Atiyah and Todd~\cite{AT_1960} for $\tilde{K}^j(\mathbb{C}P^n_k)$.) 
Let $\gamma$ be the canonical complex line bundle
over $\mathbb{C}P^n=\mathbb{C}P_1^n$.
Let $\mathbb{Z}[a]/(a^r)$ denote the truncated polynomial algebra
generated by an element $a$.
Then it is well known that
\begin{align*}
H^*(\mathbb{C}P^n)&=\mathbb{Z}[t]/(t^{n+1}), \\
K(\mathbb{C}P^n)&=\mathbb{Z}[z]/(z^{n+1}),
\end{align*}
and $\tilde{K}^{-1}(\mathbb{C}P^n)=0$
where $t\in H^2(\mathbb{C}P^n)$
and $z=\gamma-1\in\tilde{K}^0(\mathbb{C}P^n)$.
Also it is well known that
\[
 \ch(z)=\ch(\gamma)-1
=e^{-t}-1
=\sum_{r=1}^{n}(-1)^r\frac{t^r}{r!}
\in \bigoplus_{r=0}^{n}H^{2r}(\mathbb{C}P^{n}; \mathbb{Q})
\]
for the suitably chosen generator
$t\in H^2(\mathbb{C}P^n)\subset 
H^2(\mathbb{C}P^n; \mathbb{Q})$.
Then we have
\begin{equation}
\label{Chern}
  \ch(z^j)=(e^{-t}-1)^j
=\sum_{r=j}^{n}(-1)^rj!S(r, j) \frac{t^r}{r!}
=\sum_{r=j}^{n}(-1)^r\tilde{S}(r, j) \frac{t^r}{r!}
\end{equation}
(see Quaintance and Gould~\cite[(9.59)]{QG_2016})
and hence we have 
\[
\iota(\Theta(z^j))=(-1)^{n}n!\ch_n(z^j)
=(-1)^{n}n!(-1)^n\tilde{S}(n, j)\frac{t^n}{n!}=\tilde{S}(n, j)t^n
\]
for $1\leq j\leq n$.

Next, let 
$q=q_{j, \ell}\colon \mathbb{C}P^n_j\to\mathbb{C}P^n_{\ell}$
be the natural projection for $j\leq \ell$.
Let $\mathbb{Z}\{a_1, a_2, \cdots\}$ denote the free additive group
generated by elements $a_1, a_2, \cdots$.
Then, it is well known that
\begin{align*}
H^*(\mathbb{C}P_k^n)&=\mathbb{Z}\{t^k,\ t^{k+1},\ \cdots,\ t^n\}, \\
\tilde{K}^0(\mathbb{C}P^n_k)
&=\mathbb{Z}\{z^k,\ z^{k+1},\ \cdots,\ z^n\}, 
\end{align*}
and $\tilde{K}^{-1}(\mathbb{C}P_k^n)=0$.
Here the elements which correspond to
$t^j\in H^{2j}(\mathbb{C}P^n)$ and
$z^j\in \tilde{K}^0(\mathbb{C}P^n)$
through
$q_{1, k}^*\colon H^{2j}(\mathbb{C}P_k^n)
\to H^{2j}(\mathbb{C}P^n)$
and
$q_{1, k}^*\colon \tilde{K}^0(\mathbb{C}P_k^n)$
$\to\tilde{K}^0(\mathbb{C}P^n)$
are denoted
by the same letters
$t^j\in H^{2j}(\mathbb{C}P_k^n)$
and
$z^j\in \tilde{K}^0(\mathbb{C}P_k^n)$
respectively for $k\leq j\leq n$.

Then,
by \eqref{HKseq},
we have an exact sequence
\[
 \tilde{K}^0(\mathbb{C}P^n_k)\xrightarrow{\Theta}
H^{2n}(\mathbb{C}P^n_k)
\to U_{n}(\mathbb{C}P^n_k)\to 0
\]
of groups and
by the naturality of $\ch$,
we have
\[
 \iota(\Theta(z^j))=\tilde{S}(n, j)t^n
\]
for $k\leq j\leq n$.
It follows that
$\mathrm{Im}\Theta$ is the subgroup generated by 
$\tilde{S}(n, j)t^n$ for $j=k, k+1, \cdots, n$,
that is,
the subgroup generated by $\Delta_{n, k}t^n$
in $H^{2n}(\mathbb{C}P^n_k)=\mathbb{Z}\{t^n\}$.
Thus,
we have the following proposition.
\begin{proposition}
\label{gcd_is_order}
Let $n$ and $k$ be positive integers
such that $k\leq n$.
Then, the group $U_n(\mathbb{C}P^n_k)$ is
isomorphic to a cyclic group of order $\Delta_{n, k}$.
Moreover, there exists a sequence of epimorphisms of groups
\begin{equation}
\label{Useq_again}
U_n(\mathbb{C}P^{n}_{n})
\xrightarrow{q^*} U_n(\mathbb{C}P^{n}_{n-1})
\xrightarrow{q^*} \cdots
\xrightarrow{q^*} U_n(\mathbb{C}P^{n}_{2})
\xrightarrow{q^*} U_n(\mathbb{C}P^{n}_{1})
=0.
\end{equation}
\end{proposition}
Hence, Theorem~{\rm\ref{unstable}} follows from
this proposition and Corollary~{\rm\ref{gcd}}.

As an example of Theorem~{\rm\ref{unstable}}, 
for an odd integer $n$ such that $n\geq 3$,
\[
 \pi_{2n}(U(n))=U_n(S^{2n})
=U_n(\mathbb{C}P^n_{n})\cong U_n(\mathbb{C}P^n_{n-1})
\cong\mathbb{Z}/n!\mathbb{Z}
\]
by the equality $\Delta_{n, n}=n!$
or by the theorem of Borel and Hirzebruch~\cite{BH_1959}
for the homotopy group $\pi_{2n}(U(n))$.
As another example, 
for an odd integer $n$ such that $n\geq 3$,
\[
 U_n(\mathbb{C}P^n_3)\cong U_n(\mathbb{C}P^n_2)\cong
\mathbb{Z}/B'_{n-1}\mathbb{Z}
\]
by the theorem of Komatsu, Luca, and Ruiz~\cite{KLR_2014}
which states that $\Delta_{n, 2}=B'_{n-1}$.

\begin{corollary}
\label{power_of_two}
Let $n$ be an integer which is a power of two.
Then,
the $2$-primary component of 
$U_n(\mathbb{C}P^n_k)$ is a cyclic group 
of order $2^{k-1}$.
\end{corollary}
\begin{proof}
In Wannemacker~\cite{Wannemacker_2005},
it is shown that $\nu_2(\tilde{S}(n, k))=k-1$
if $n$ is a power of two.
Then, the corollary follows from Proposition~{\rm\ref{gcd_is_order}}.
\end{proof}
Let $\Sigma X$ denote the reduced suspension of a topological space $X$.
Let $b\colon S^{2k}\to \mathbb{C}P^n_k$ 
be the inclusion of the bottom cell of $\mathbb{C}P^n_k$
and put
\[
 \kappa_{n, k}=\frac{\Delta_{n, k+1}}{\Delta_{n, k}}
\]
for positive integers $n$ and $k$ such that $k<n$.
\begin{proposition}
\label{suspension}
Let $n$ and $k$ be positive integers
such that $k< n$.
Then, the group $U_n(\Sigma \mathbb{C}P^n_k)$ is 
isomorphic to a free aditive group of rank $n-k$.
Moreover, there exists a sequence of monomorphisms of groups
\begin{equation*}
\label{Useq_mono} 
U_n(\Sigma \mathbb{C}P^{n}_{n-1})
\xrightarrow{(\Sigma q)^*} \cdots
\xrightarrow{(\Sigma q)^*} U_n(\Sigma \mathbb{C}P^{n}_{2})
\xrightarrow{(\Sigma q)^*} U_n(\Sigma \mathbb{C}P^{n}_{1})
\end{equation*}
and the image of
\[
(\Sigma b)^*\colon 
U_n(\Sigma \mathbb{C}P^n_k)\to U_n(S^{2k+1})=\pi_{2k+1}(U(n))\cong
 \mathbb{Z}
\]
is $\kappa_{n, k} U_n(S^{2k+1})$.
\end{proposition}
\begin{proof}
We have an exact sequence of groups
\[
 0\to U_n(\Sigma \mathbb{C}P^n_k)
\to \tilde{K}^{-1}(\Sigma \mathbb{C}P^n_k)
\xrightarrow{T}H^{2n+1}(\Sigma \mathbb{C}P^n_k)
\]
by Theorem 1.2 and Theorem 1.3 of Hamanaka~\cite{Hamanaka_2004},
by $H^{2n}(\Sigma \mathbb{C}P^n_k)=0$,
and by $Sq^2\rho H^{2n-1}(\Sigma\mathbb{C}P^n_k)
=H^{2n+1}(\Sigma\mathbb{C}P^n_k; \mathbb{Z}/2\mathbb{Z})$
if $n$ is even where
\[
\rho\colon H^{2n-1}(\Sigma\mathbb{C}P^n_k)
\to H^{2n-1}(\Sigma\mathbb{C}P^n_k; \mathbb{Z}/2\mathbb{Z})
\]
is the mod $2$ reduction
and
\[
Sq^2\colon H^{2n-1}(\Sigma\mathbb{C}P^n_k; \mathbb{Z}/2\mathbb{Z})
\to H^{2n+1}(\Sigma\mathbb{C}P^n_k; \mathbb{Z}/2\mathbb{Z})
\]
is the Steenrod squaring operation.
Moreover, 
the homomorphism 
\[
T\colon \tilde{K}^{-1}(\Sigma \mathbb{C}P^n_k)
\to H^{2n+1}(\Sigma \mathbb{C}P^n_k)
\]
is identified with
\[
 \Theta\colon \tilde{K}^{0}(\mathbb{C}P^n_k)
\to H^{2n}(\mathbb{C}P^n_k)
\]
through the suspension isomorphism
\[
\sigma^*\colon H^{2n}(\mathbb{C}P^n_k)\to 
H^{2n+1}(\Sigma \mathbb{C}P^n_k)
\]
and the isomorphism
\[
 \tilde{K}^{0}(\mathbb{C}P^n_k)
=[\mathbb{C}P^n_k, BU]_*
\xrightarrow{\beta_*}
[\mathbb{C}P^n_k, \Omega U]_*
\cong
[\Sigma\mathbb{C}P^n_k, U]_*
=\tilde{K}^{-1}(\Sigma \mathbb{C}P^n_k)
\]
where $\beta\colon BU\to \Omega U$ is the Bott map.
It follows that 
$U_n(\Sigma \mathbb{C}P^n_k)\cong \textrm{Ker}\ T$
is isomorphic to a free additive group of rank $n-k$.

Next, for $k\leq n-2$,
considering the following homotopy cofibration sequence
\[
 S^{2k}\xrightarrow{b}\mathbb{C}P^{n}_{k}
\xrightarrow{q}\mathbb{C}P^{n}_{k+1}
\to S^{2k+1}\xrightarrow{\Sigma b}\Sigma\mathbb{C}P^{n}_{k}
\xrightarrow{\Sigma q}\Sigma\mathbb{C}P^{n}_{k+1}
\to S^{2k+2}
\]
and applying the functor $U_n(-)$,
we have the following exact sequence of groups
\begin{align*}
 0\to &U_n(\Sigma \mathbb{C}P^{n}_{k+1})
 \xrightarrow{(\Sigma q)^*}
 U_n(\Sigma \mathbb{C}P^{n}_{k})
 \xrightarrow{(\Sigma b)^*}
 U_n(S^{2k+1}) \\
&\qquad \to U_n(\mathbb{C}P^{n}_{k+1})
\xrightarrow{q^*}
 U_n(\mathbb{C}P^{n}_{k})
 \xrightarrow{b^*}
 0.
\end{align*}
Thus, by Proposition~\ref{gcd_is_order},
the proposition follows.
\end{proof}
\noindent
\textbf{Remark.}
\ As is well known, 
$U_n(\Sigma \mathbb{C}P^{n}_{n})=\pi_{2n+1}(U(n))$
is trivial if $n$ is odd
and is isomorphic to $\mathbb{Z}/2\mathbb{Z}$ if $n$ is even.
See Toda~\cite{Toda_1959}.
\begin{corollary}
 For any odd integer $n\geq 3$ and any positive even integer $k$
such that $k<n$, 
the sequence of groups
\[
 0\to U_n(\Sigma \mathbb{C}P^{n}_{k+1})
 \xrightarrow{(\Sigma q)^*}
 U_n(\Sigma \mathbb{C}P^{n}_{k})
 \xrightarrow{(\Sigma b)^*}
 U_n(S^{2k+1})\to 0
\]
is split exact.
\end{corollary}
\begin{proof}
This follows from Proposition~\ref{suspension}
and Corollary~\ref{gcd}.
\end{proof}
\noindent
\textbf{Remark.}
\ We can give generators of 
\begin{align*}
  U_n(\Sigma \mathbb{C}P^n_k)&\cong 
 \textrm{Ker} (T\colon \tilde{K}^{-1}(\Sigma \mathbb{C}P^n_k)
 \to H^{2n+1}(\Sigma \mathbb{C}P^n_k)) \\
 &\cong \textrm{Ker} (\Theta\colon \tilde{K}^{0}(\mathbb{C}P^n_k)
\to H^{2n}(\mathbb{C}P^n_k)) \\
&\cong \mathbb{Z}^{n-k} 
\end{align*}
using arithmetic relations of $\tilde{S}(n,k)$.
As an example, for $n=7$,
we have 
\begin{align*}
&\phantom{=}
(\tilde{S}(7, 7), \tilde{S}(7, 6), \tilde{S}(7, 5), \tilde{S}(7, 4), 
\tilde{S}(7, 3), \tilde{S}(7, 2), \tilde{S}(7, 1)) \\
&=(5040, 15120, 16800, 8400, 1806, 126, 1).
\end{align*}
Hence, we have 
\[
 (\Delta_{7, 7}, \Delta_{7, 6}, \Delta_{7, 5}, \Delta_{7, 4}, 
\Delta_{7, 3}, \Delta_{7, 2}, \Delta_{7, 1})=(5040, 5040, 1680, 1680, 42, 42, 1)
\]
and 
\[
(\kappa_{7, 6}, \kappa_{7, 5}, \kappa_{7, 4}, 
\kappa_{7, 3}, \kappa_{7, 2}, \kappa_{7, 1})
=(1, 3, 1, 40, 1, 42).
\]
Define elements $\alpha_6, \cdots, \alpha_1$ by
\begin{align*}
&\phantom{=}
\begin{pmatrix}
\alpha_6 & \alpha_5 & \alpha_4 & \alpha_3 & \alpha_2 & \alpha_1
\end{pmatrix} \\
&=
\begin{pmatrix}
z^7 & z^6 & z^5 & z^4 & z^3 & z^2 & z
\end{pmatrix}
\begin{pmatrix}
-3 & -10 & -5 & -16 &  2 & -1 \\
 1 &   0 &  0 &   0 &  0 &  0 \\
 0 &   3 &  1 &   0 &  0 &  0 \\
 0 &   0 &  1 &   1 & -1 &  0 \\
 0 &   0 &  0 &  40 & -1 &  2 \\
 0 &   0 &  0 &   0 &  1 & 11 \\
 0 &   0 &  0 &   0 &  0 & 42 
\end{pmatrix}.
\end{align*}
Then, for $k=6, 5, \cdots, 1$, we can easily see that
$\alpha_6, \cdots, \alpha_k$
generates 
\[
\textrm{Ker} (\Theta\colon \tilde{K}^{0}(\mathbb{C}P^7_k)
\to H^{14}(\mathbb{C}P^7_k)) \\
\cong \mathbb{Z}^{7-k},
\] 
which is a subgroup of
\[
 \tilde{K}^0(\mathbb{C}P^7_k)=\mathbb{Z}\{z^7,\  \cdots,\ z^k\}.
\]

Finally, 
we give a characterization of $b_k=M_k$
in terms of the sequence \eqref{Useq_again}.
Recall that for a positive integer $k$,
there exists a positive integer $b_k$,
the complex James number
of Stiefel manifolds, 
such that
the following conditions are equivalent:
\begin{description}
\item[(A)]
the integer $n$ is a positive multiple of $b_k$
 \item[(B)] 
the fibration
\[
W_{n, k}=U(n)/U(n-k) \to W_{n, 1}=U(n)/U(n-1)=S^{2n-1}
\]
of complex Stiefel manifolds admits a cross-section.
\end{description}
(See James~\cite{James_1958} and Atiyah~\cite{Atiyah_1961}.)
Also recall that for a positive integer $k$,
the Atiyah-Todd number $M_k$ is defined by 
\[
\nu_p(M_k)=
\begin{cases}
 &\max\left\{r+\nu_p(r)\ \big\vert\ 1\leq r\leq
 \left[\frac{k-1}{p-1}\right]\right\}
 \quad \text{if}\ \ p\leq k, \\
 &0 \quad \text{if}\ \ k<p.
\end{cases}
\]
(See Atiyah and Todd~\cite{AT_1960}.)
For example, $M_1=1,\ M_2=2,\ M_3=M_4=24,\ M_5=M_6=2880, \cdots$.
It is obvious that $M_k$ divides $M_{k+1}$
and it is shown in \cite{AT_1960} that $M_k=M_{k+1}$
for any odd integer $k\geq 3$.
In \cite{AT_1960},
it is shown that $b_k$ is a positive multiple of $M_k$
for any positive integer $k$.
Later, Adams and Walker
show that $b_k=M_k$
for any positive integer $k$
by examining the $J$-group of $\mathbb{C}P^{k-1}$.
(See \cite{AW_1965}. 
Also see \cite{Atiyah_1961}.)
Also it is known that there are several conditions 
which are equivalent to \textbf{(A)} and \textbf{(B)}.
For example, 
in terms of stable homotopy theory,
the condition that
\begin{description}
 \item[(C)] 
the stunted complex projective space $\mathbb{C}P^{n-1}_{n-k}$
is $S$-reducible, 
that is,
the stable homotopy class of the attaching map of the top cell 
of $\mathbb{C}P^{n-1}_{n-k}$ is inessential
\end{description}
is equivalent to \textbf{(A)} and \textbf{(B)}.
(See James~\cite{James_1959}.)

Now, 
by the result of Atiyah and Todd~\cite[Theorem~(1.7)]{AT_1960},
the condition that
$n$ is a positive multiple of $M_{k}$
is equivalent to that
the coefficient of $t^{n-1}$ in $(e^t-1)^j$
is an integer for $j=n-k, n-k+1, \cdots, n-1$.
By \eqref{Chern},
this condition is equivalent to that
$\tilde{S}(n-1, j)$ is a positive multiple 
of $(n-1)!$
for $j=n-k, n-k+1, \cdots, n-1$ and hence,
is equivalent to that
$\Delta_{n-1, n-k}=(n-1)!$.
Thus we have another condition
which is equivalent to \textbf{(A)}, 
\textbf{(B)}, and \textbf{(C)}
as in the following proposition.
\begin{proposition}
\label{James}
The integer $n$ is a positive multiple 
of the complex James number and
the Atiyah-Todd number $b_{k}=M_{k}$
if and only if
we have isomorphisms of groups
\[
U_{n-1}(\mathbb{C}P^{n-1}_{n-1})\cong
 U_{n-1}(\mathbb{C}P^{n-1}_{n-2})\cong 
\cdots \cong 
U_{n-1}(\mathbb{C}P^{n-1}_{n-k})\cong\mathbb{Z}/(n-1)!\mathbb{Z}.
\]
\end{proposition}
For example,
since $24$ is a positive multiple 
of $b_{4}=M_{4}=24$,
we have
\[
\Delta_{23, 23}=
\Delta_{23, 22}=
\Delta_{23, 21}=
\Delta_{23, 20}=23!
\]
and
\[
U_{23}(\mathbb{C}P^{23}_{23})\cong U_{23}(\mathbb{C}P^{23}_{22})\cong
U_{23}(\mathbb{C}P^{23}_{21})\cong U_{23}(\mathbb{C}P^{23}_{20})\cong
\mathbb{Z}/23!\mathbb{Z}.
\]


\begin{thebibliography}{99}

\bibitem{AW_1965}
J. F. Adams and G. Walker, 
On complex Stiefel manifolds, 
\textit{Proc. Cambridge Philos. Soc.} \textbf{61} (1965), 81--103.
%
\bibitem{Atiyah_1961}
M. F. Atiyah, 
Thom complexes, 
\textit{Proc. London Math. Soc.} \textbf{s3-11} (1961), 291--310.
%
\bibitem{AT_1960}
M. F. Atiyah and J. A. Todd, 
On complex Stiefel manifolds, 
\textit{Proc. Cambridge Philos. Soc.} \textbf{56} (1960), 342--353.
%
\bibitem{BD_1991}
M. Bendersky and D. M. Davis, 
$2$-primary $v_1$-periodic homotopy groups of $SU(n)$, 
\textit{Amer. J. Math.} \textbf{114} (1991), 465--494.
%
\bibitem{BH_1959}
A. Borel and F. Hirzebruch, 
Characteristic classes and homogeneous spaces II, 
\textit{Amer. J. Math.} \textbf{81} (1959), 315--382.
%
\bibitem{Clausen_1840}
T. Clausen, 
Theorem, 
\textit{Astron. Nach.} \textbf{17} (1840), 351--352.
%
\bibitem{CK_1988}
M. C. Crabb and K. Knapp, 
The Hurewicz map on stunted complex
projective spaces, 
\textit{Amer. J. Math.} \textbf{110} (1988), 783--809.
%
\bibitem{Davis_1991}
D. M. Davis, 
The $v_1$-periodic homotopy groups of $SU(n)$ at odd primes, 
\textit{J. London Math. Soc.} \textbf{43} (1991), 529--544.
%
\bibitem{Davis_2008}
D. M. Davis, 
Divisibility by $2$ and $3$ of certain Stirling numbers, 
\textit{Integers} \textbf{8} (2008), A56.
%
\bibitem{DP_2007}
D. M. Davis and K. Potocka, 
$2$-primary $v_1$-periodic homotopy groups of $SU(n)$ revisited, 
\textit{Forum Math.}
\textbf{19} (2007), 783--822.
%
\bibitem{Hamanaka_2004}
H. Hamanaka, 
On $[X, U(n)]$ when $\dim X$ is $2n+1$, 
\textit{J. Math. Kyoto Univ.} \textbf{44-3} (2004), 655--667.
%
\bibitem{HK_2003}
H. Hamanaka and A. Kono, 
On $[X, U(n)]$ when $\dim X$ is $2n$, 
\textit{J. Math. Kyoto Univ.} \textbf{43-2} (2003), 333--348.
%
\bibitem{HK_2006}
H. Hamanaka and A. Kono, 
Unstable $K^1$-group and homotopy type
of certain gauge groups, 
\textit{Proc. Royal Soc. Edinburgh} \textbf{136A} (2006), 149--155.
%
\bibitem{HK_2007}
H. Hamanaka and A. Kono, 
Homotopy type of gauge groups
of $SU(3)$-bundles over $S^6$, 
\textit{Topology and its Applications} \textbf{154} (2007), 1377--1380.
%
\bibitem{James_1958}
I. M. James, 
Cross-sections of Stiefel manifolds, 
\textit{Proc. London Math. Soc.} \textbf{s3-8} (1958), 536--547.
%
\bibitem{James_1959}
I. M. James, 
Spaces associated with Stiefel manifolds, 
\textit{Proc. London Math. Soc.} \textbf{s3-9} (1959), 115--140.
%
\bibitem{KLR_2014}
T. Komatsu, F. Luca and C. Ruiz, 
A note on the denominators of Bernoulli numbers, 
\textit{Proc. Japan Acad.}
\textbf{90} (2014), 71--72.
%
\bibitem{Lundell_1974}
A. T. Lundell, 
Generalized $e$-invariants and the numbers of James, 
\textit{Quart. J. Math. Oxford}
\textbf{25} (1974), 427--440.
%
\bibitem{Lundell_1978}
A. T. Lundell, 
A divisibility property for Stirling numbers, 
\textit{J. Number Theory} \textbf{10} (1978), 35--54.
%
\bibitem{QG_2016}
J. Quaintance and H. W. Gould, 
\textit{Combinatorial identities for Stirling numbers 
{\rm (}The unpublished notes of H. W. Gould {\rm )}} 
(World Scientific , 2016).
%
\bibitem{Staudt_1840}
K. G. C. von Staudt, 
Beweis eines Lehrsatzes, 
die Bernoullischen Zahlen betreffend,  
\textit{J. reine angew. Math.} \textbf{21} (1840), 372--374.
%
\bibitem{Toda_1959}
H. Toda,
A topological proof of theorems of Bott and
Borel-Hirzebruch for homotopy groups of unitary groups,
\textit{Mem. Coll. Sci. Univ. Kyoto} \textbf{32} (1959), 103-119.
%
\bibitem{Wannemacker_2005}
S. Wannemacker, 
On $2$-adic orders of Stirling numbers of the second kind, 
\textit{Integers} \textbf{5(1)} (2005), A21.
%
\end{thebibliography}
\end{document}